\documentclass{amsart}
\usepackage{amssymb}
\usepackage{mathrsfs}
\usepackage{amsmath}
\usepackage{amsfonts}
\usepackage{pifont}
\usepackage{graphicx,amssymb,mathrsfs,amsmath}
\usepackage{tocvsec2}
\usepackage{color}
\newtheorem{thm}{Theorem}[section]

\newtheorem{prop}[thm]{Proposition}
\theoremstyle{definition}
\newtheorem{defn}[thm]{Definition}
\newtheorem{example}[thm]{Example}
\theoremstyle{remark}
\newtheorem{rem}[thm]{Remark}
\numberwithin{equation}{section}

\begin{document}
\title[Distributional chaos and Li-Yorke chaos in metric spaces]{Distributional chaos and Li-Yorke chaos in metric spaces}

\author{Marko Kosti\' c}
\address{Faculty of Technical Sciences,
University of Novi Sad,
Trg D. Obradovi\' ca 6, 21125 Novi Sad, Serbia}
\email{marco.s@verat.net}

{\renewcommand{\thefootnote}{} \footnote{
\\ \text{  }  \ \    {\it Key words and phrases.} Distributional chaos; Li-Yorke chaos; binary relations; metric spaces; multivalued linear operators.
\\  \text{  }  \ \ The author is partially supported by grant 174024 of Ministry of Science and Technological Development, Republic of Serbia.}}

\begin{abstract}
In this paper, we introduce several new types and generalizations of the concepts distributional chaos and Li-Yorke chaos.
We consider the general sequences of binary relations acting between metric spaces, while in a separate section we focus our attention to some special features of distributionally chaotic and Li-Yorke chaotic multivalued linear operators in Fr\' echet spaces.
\end{abstract}
\maketitle

\section{Introduction and Preliminaries}\label{intro}

Let $X$ be a separable Fr\' echet space. A
linear operator $T$ on $X$ is said to be hypercyclic iff there
exists an element $x\in D_{\infty}(T)\equiv \bigcap_{n\in {\mathbb
N}}D(T^{n})$ whose orbit $\{ T^{n}x : n\in {{\mathbb N}}_{0} \}$ is
dense in $X;$ $T $ is said to be topologically transitive,
resp. topologically mixing, iff for every pair of open
non-empty subsets $U,\ V$ of $X,$ there exists $n_{0}\in {\mathbb
N}$ such that $T^{n_{0}}(U) \ \cap \ V \neq \emptyset ,$ resp. there
exists $n_{0}\in {\mathbb N}$ such that, for every $n\in {\mathbb
N}$ with $n\geq n_{0},$ $T^{n}(U) \ \cap \ V \neq \emptyset .$ A
linear operator $T$ on $X$ is said to be chaotic iff it is topologically transitive and the set of periodic points of $T,$ defined by $\{x\in D_{\infty}(T) : (\exists n\in {\mathbb N})\, T^{n}x=x\},$ is dense in $X.$

The basic facts about topological dynamics of linear continuous operators in Banach and Fr\' echet spaces can be obtained by consulting the monographs \cite{bayart} by F. Bayart, E. Matheron and
\cite{erdper} by K.-G. Grosse-Erdmann, A. Peris. In a joint research study with C.-C. Chen, J. A. Conejero and M. Murillo-Arcila \cite{kerry-drew}, the author has recently introduced and analyzed a great deal of topologically dynamical properties for multivalued linear operators (cf. also the article \cite{abakumov} by E. Abakumov, M. Boudabbous and  M. Mnif). The notion has been extended in \cite{kd-prim}  and \cite{marek-faac} for general sequences of binary relations over topological spaces. 

Distributional chaos for interval maps was introduced by B. Schweizer and J. Sm\' ital in \cite{smital} (this type of chaos was called strong chaos there, 1994). For linear continuous operators, 
distributional chaos was firstly investigated in the research studies of quantum harmonic oscillator, by J. Duan et al \cite{duan} (1999) and P. Oprocha \cite{p-oprocha} (2006). Distributional chaos for linear continuous operators in Fr\' echet spaces was analyzed by 
N. C. Bernardes Jr. et al \cite{2013JFA} (2013), while distributional chaos for closed linear operators in Fr\' echet spaces was investigated by J. A. Conejero et al \cite{mendoza} (2016). 

On the other hand, the notion of Li-Yorke chaos received enormous attention after the foundational paper of T. Y. Li and J. A. Yorke \cite{LY} (1975). Li-Yorke chaotic linear continuous operators on
Banach and Fr\' echet spaces have been systematically analyzed in \cite{2011} and \cite{band}. For more details about Li-Yorke chaos and distributional chaos  in metric and Fr\' echet spaces, we refer the reader to \cite{akin}, \cite{blan}, \cite{fu-wang}-\cite{huang}, \cite{nis-dragan}, \cite{luo}-\cite{p-oprocha-tams}, \cite{tan-fu} and references cited therein.

The main aim of this paper is to introduce various notions of distributional chaos and Li-Yorke chaos for binary relations and their sequences, working in the setting of metric spaces. In particular, we analyze the notions of reiteratively $\tilde{X}$-distributional chaos of types $1$ and $2$,
$
(\tilde{X},i)$-mixed chaos, where $i=1,2,3,4,$ and $
\tilde{X}$-Li-Yorke chaos; here, $\tilde{X}$ is a non-empty subset of metric space $X$ under our consideration.
In \cite{2018JMMA}, N. C. Bernardes Jr. et al have considered the notion of distributional chaos of type $s$ for linear continuous operators acting on Banach spaces ($s\in \{1,2,2\frac{1}{2},3\}$). In this paper, we extend the notion from this paper for general sequences of binary relations, as well (up to now, the notion from \cite{2018JMMA} has not been introduced  for linear unbounded operators on Banach spaces and their sequences). Finally, we analyze distributionally chaotic and Li-Yorke chaotic multivalued linear operators in Fr\' echet spaces by enquiring into the basic properties of associated distributionally chaotic and Li-Yorke chaotic irregular vectors and their submanifolds. Plenty of useful comments, observations and open problems enriches our study.

Albeit defined in this general framework, 
we feel duty bound to say that the notion of distributional chaos, its specifications and various generalizations are the most intriguing for orbits of continuous linear operators in Banach spaces. This follows from a series of simple counterexamples presented in this paper, which show in particular that distributional chaos and Li-Yorke chaos occur even for the general sequences of continuous linear operators on finite-dimensional spaces. All aspects and connections between the introduced concepts cannot be easily perceived within just one research paper and our intention here, actually, was to create a solid base for further explorations of distributional chaos and Li-Yorke chaos. Because of that, we can freely say that this paper is heuristic to a large extent.

The organization of material is briefly described as follows. In Subsection \ref{publi} and Subsection \ref{MLOs}, we recall the basic things about lower and upper densities as well as 
binary relations and multivalued linear operators, respectively. In Section \ref{marek-MLOs} and Section \ref{marek-MLOsk}, we analyze various types of distributional chaos, Li-Yorke chaos and distributional chaos of type $s$ ($s\in \{1,2,2\frac{1}{2},3\}$) for binary relations over metric spaces. The fourth section of paper is reserved for the study of 
distributional chaos and Li-Yorke chaos for multivalued linear operators in Fr\' echet spaces; in a separate subsection, we investigate irregular vectors and irregular manifolds. In addition to the above, 
we include the conclusion and remark section at the end of paper.

Before proceeding further, we need to recall that for each
 set $D=\{d_{n} : n\in {\mathbb N}\}$, where $(d_{n})_{n\in {\mathbb N}}$ is a strictly increasing sequence of positive integers, we define its complement $D^{c}:={\mathbb N} \setminus D$ and difference set $\{e_{n}:=d_{n+1}-d_{n} | \, n\in {\mathbb N}\}.$ Let us recall that an infinite subset $A
$ of ${\mathbb N}$ is said to be syndetic, or relatively dense, iff its difference set is bounded. The difference set of any finite subset of ${\mathbb N},$ defined similarly as above, is finite. Set 
${\mathbb N}_{n}:=\{1,\cdot \cdot \cdot,n\}$ ($n\in {\mathbb N}$) and $S_{1}:=\{z\in {\mathbb C}\, ; |z|=1\}.$ By $P(A)$ we denote the power set of $A$. 

\subsection{Lower and upper densities}\label{publi}

In this subsection, we recall the basic things about lower and upper densities that will be necessary for our further work.

Let $A\subseteq {\mathbb N}$ be non-empty. The lower density of $A,$ denoted by $\underline{d}(A),$ is defined by
$$
\underline{d}(A):=\liminf_{n\rightarrow \infty}\frac{|A \cap [1,n]|}{n},
$$
and the upper density of $A,$ denoted by $\overline{d}(A),$ is defined by
$$
\overline{d}(A):=\limsup_{n\rightarrow \infty}\frac{|A \cap [1,n]|}{n}.
$$
Further on, the lower Banach density of $A,$ denoted by $\underline{Bd}(A),$ is defined by
$$
\underline{Bd}(A):=\lim_{s\rightarrow +\infty}\liminf_{n\rightarrow \infty}\frac{|A \cap [n+1,n+s]|}{s}
$$
and the (upper) Banach density of $A,$ denoted by $\overline{Bd}(A),$ is defined by
$$
\overline{Bd}(A):=\lim_{s\rightarrow +\infty}\limsup_{n\rightarrow \infty}\frac{|A \cap [n+1,n+s]|}{s}.
$$
It is well known that the limits appearing in definitions of $
\underline{Bd}(A)$ and $
\overline{Bd}(A)$ exist as $s$ tends to $+\infty,$ as well as that
\begin{align}\label{jebu}
0\leq \underline{Bd}(A) \leq \underline{d}(A) \leq \overline{d}(A) \leq \overline{Bd}(A)\leq 1,
\end{align}
\begin{align}\label{jebu1}
\underline{d}(A) + \overline{d}(A^{c})=1
\end{align}
and
\begin{align}\label{jebu2}
\underline{Bd}(A) +\overline{Bd}(A^{c})=1.
\end{align}

\subsection{Binary relations and multivalued linear operators}\label{MLOs}

Let $X,\ Y,\ Z$ and $ T$ be given non-empty sets. A binary relation between $X$ into $Y$
is any subset
$\rho \subseteq X \times Y.$ 
If $\rho \subseteq X\times Y$ and $\sigma \subseteq Z\times T$ with $Y \cap Z \neq \emptyset,$ then
we define $\rho^{-1} \subseteq Y\times X$
and
$\sigma \circ \rho \subseteq X\times T$ by
$
\rho^{-1}:=\{ (y,x)\in Y\times X : (x,y) \in \rho \}
$
and
$$
\sigma \circ \rho :=\bigl\{(x,t) \in X\times T : \exists y\in Y \cap Z\mbox{ such that }(x,y)\in \rho\mbox{ and }
(y,t)\in \sigma \bigr\},
$$
respectively. Domain and range of $\rho$ are introduced by $D(\rho):=\{x\in X :
\exists y\in Y\mbox{ such that }(x,y)\in \rho \}$ and $R(\rho):=\{y\in Y :
\exists x\in X\mbox{ such that }(x,y)\in \rho\},$ respectively; $\rho (x):=\{y\in Y : (x,y)\in \rho\}$ ($x\in X$), $ x\ \rho \ y \Leftrightarrow (x,y)\in \rho .$
If $\rho$ is a binary relation on $X$ and $n\in {\mathbb N},$ then we define $\rho^{n}
$ inductively; $\rho^{-n}:=(\rho^{n})^{-1}$ and $\rho^{0}:=
\{(x,x) : x\in X\}.$ Put $D_{\infty}(\rho):=\bigcap_{n\in {\mathbb N}} D(\rho^{n})$ and $\rho (X'):=\{y : y\in \rho(x)\mbox{ for some }x\in X'\}$ ($X'\subseteq X$). 

In the remaining part of this subsection, we present a brief overview
of the necessary definitions and properties of multivalued linear operators\index{multivalued linear operator!MLO}. For more details about the subject, we refer the reader to the monographs  \cite{cross} by R. Cross and \cite{faviniyagi} by A. Favini, A. Yagi (in \cite{faviniyagi}, applications of multivalued linear operators to abstract degenerate differential equations have been thoroughly analyzed; for some other approaches, the reader may consult the monograph \cite{svir-fedorov} by G. A. Sviridyuk and V. E. Fedorov).

Let $X$ and $Y$ be two Fr\' echet spaces over the same field of scalars ${\mathbb K}.$ For any mapping ${\mathcal A}: X \rightarrow P(Y)$ we define $\check{{\mathcal A}}:=\{(x,y) : x\in D({\mathcal A}),\ y\in {\mathcal A}x\}.$ Then ${\mathcal A}$ is a multivalued linear operator (MLO) iff the associated binary relation $\check{{\mathcal A}}$ is a linear relation in $X\times Y,$ i.e., iff $\check{{\mathcal A}}$ is a linear subspace of $X \times Y.$ In our work, we will identify ${\mathcal A}$ and its associated linear relation $\check{{\mathcal A}},$ so that the notion of $D({\mathcal A}),$ which is a linear subspace of $X,$ as well as the sets $R({\mathcal A})$
and $D_{\infty}({\mathcal A})$ are clear. The set ${\mathcal A}^{-1}0 = \{x \in D({\mathcal A}) : 0 \in {\mathcal A}x\}$ is called the kernel\index{multivalued linear operator!kernel}
of ${\mathcal A}$ and it is denoted henceforth by $N({\mathcal A})$ or Kern$({\mathcal A}).$ The inverse ${\mathcal A}^{-1}$ and the power ${\mathcal A}^{n}$ of a MLO, introduced in the sense of corresponding definition for general binary relations, are MLOs ($n\in {\mathbb N}$).
If $X=Y,$ then we say that ${\mathcal A}$ is an MLO in $X.$
An almost immediate consequence of definition is that,
for every $x,\ y\in D({\mathcal A})$ and $\lambda,\ \eta \in {\mathbb K}$ with $|\lambda| + |\eta| \neq 0,$ we
have $\lambda {\mathcal A}x + \eta {\mathcal A}y = {\mathcal A}(\lambda x + \eta y).$ If ${\mathcal A}$ is an MLO, then ${\mathcal A}0$ is a linear manifold in $Y$
and ${\mathcal A}x = f + {\mathcal A}0$ for any $x \in D({\mathcal A})$ and $f \in {\mathcal A}x.$  The sum ${\mathcal A}+{\mathcal B}$ of MLOs ${\mathcal A}$ and ${\mathcal B},$ defined by $D({\mathcal A}+{\mathcal B}) := D({\mathcal A})\cap D({\mathcal B})$ and $({\mathcal A}+{\mathcal B})x := {\mathcal A}x +{\mathcal B}x$ ($x\in D({\mathcal A}+{\mathcal B})$),
is likewise an MLO. We write ${\mathcal A} \subseteq {\mathcal B}$ iff $D({\mathcal A}) \subseteq D({\mathcal B})$ and ${\mathcal A}x \subseteq {\mathcal B}x$
for all $x\in D({\mathcal A}).$
The scalar multiplication of an MLO ${\mathcal A} : X\rightarrow P(Y)$ with the number $z\in {\mathbb K},$ $z{\mathcal A}$ for short, is defined by
$D(z{\mathcal A}):=D({\mathcal A})$ and $(z{\mathcal A})(x):=z{\mathcal A}x,$ $x\in D({\mathcal A}).$ It is clear that $z{\mathcal A}  : X\rightarrow P(Y)$ is an MLO and $(\omega z){\mathcal A}=\omega(z{\mathcal A})=z(\omega {\mathcal A}),$ $z,\ \omega \in {\mathbb K}.$ By a periodic point of ${\mathcal A}$ we mean any vector $x\in D_{\infty}({\mathcal A})$ such that  there exists $n\in {\mathbb N}$ with $x\in {\mathcal A}^{n}x.$

Suppose that ${\mathcal A}$ is an MLO in $ X.$ Then we say that a point $\lambda \in {\mathbb K}$ is an eigenvalue of ${\mathcal A}$
iff there exists a vector $x\in X\setminus \{0\}$ such that $\lambda x\in {\mathcal A}x;$ we call $x$ an eigenvector of operator ${\mathcal A}$ corresponding to the eigenvalue $\lambda.$ Observe that, if ${\mathcal A}$ is purely multivalued (i.e., ${\mathcal A}0\neq 0$), a vector $x\in X\setminus \{0\}$ can be an  eigenvector of operator ${\mathcal A}$ corresponding to different values of scalars $\lambda.$ The point spectrum of ${\mathcal A},$ $\sigma_{p}({\mathcal A})$ for short, is defined as the union of all eigenvalues of ${\mathcal A}.$

If ${\mathcal A} : X\rightarrow P(Y)$ is an MLO, then we define the adjoint ${\mathcal A}^{\ast}: Y^{\ast}\rightarrow P(X^{\ast})$\index{multivalued linear operator!adjoint}
of ${\mathcal A}$ by its graph
$$
{\mathcal A}^{\ast}:=\Bigl\{ \bigl( y^{\ast},x^{\ast}\bigr)  \in Y^{\ast} \times X^{\ast} :  \bigl\langle y^{\ast},y \bigr\rangle =\bigl \langle x^{\ast}, x\bigr \rangle \mbox{ for all pairs }(x,y)\in {\mathcal A} \Bigr\}.
$$

\section{Distributional chaos and Li-Yorke chaos for binary relations}\label{marek-MLOs}

In this section, it will be always assumed that $(X,d)$ and  $(Y,d_{Y})$ are metric spaces.
Suppose that $\sigma>0$, $\epsilon>0$ and $(x_{k})_{k\in {\mathbb N}},\ (y_{k})_{k\in {\mathbb N}}$ are two given sequences in $Y.$
Consider the following conditions:
\begin{align}\label{jednacinaj}
\begin{split}
& \underline{Bd}\Bigl( \bigl\{k \in {\mathbb N} :
d_{Y}\bigl(x_{k},y_{k}\bigr)< \sigma \bigr\}\Bigr)=0,
\\
& \underline{Bd}\Bigl(\bigl\{k \in {\mathbb N} : d_{Y}\bigl(x_{k},y_{k}\bigr)
\geq \epsilon \bigr\}\Bigr)=0,
\end{split}
\end{align}
\begin{align}\label{jednacinaj1}
\begin{split}
& \underline{Bd}\Bigl( \bigl\{k \in {\mathbb N} :
d_{Y}\bigl(x_{k},y_{k}\bigr)< \sigma \bigr\}\Bigr)=0,
\\
& \underline{d}\Bigl(\bigl\{k \in {\mathbb N} : d_{Y}\bigl(x_{k},y_{k}\bigr)
\geq \epsilon \bigr\}\Bigr)=0,
\end{split}
\end{align}
\begin{align}\label{jednacinaj2}
\begin{split}
& \underline{d}\Bigl( \bigl\{k \in {\mathbb N} :
d_{Y}\bigl(x_{k},y_{k}\bigr)< \sigma \bigr\}\Bigr)=0,
\\
& \underline{Bd}\Bigl(\bigl\{k \in {\mathbb N} : d_{Y}\bigl(x_{k},y_{k}\bigr)
\geq \epsilon \bigr\}\Bigr)=0
\end{split}
\end{align}
and
\begin{align}\label{jednacina}
\begin{split}
& \underline{d}\Bigl( \bigl\{k \in {\mathbb N} :
d_{Y}\bigl(x_{k},y_{k}\bigr)< \sigma \bigr\}\Bigr)=0,
\\
& \underline{d}\Bigl(\bigl\{k \in {\mathbb N} : d_{Y}\bigl(x_{k},y_{k}\bigr)
\geq \epsilon \bigr\}\Bigr)=0.
\end{split}
\end{align}

In the following definition, we introduce the notion of reiterative distributional chaos (reiterative distributional chaos of type $1$ or $2$):

\begin{defn}\label{DC-unbounded-fric} 
Suppose that, for every $k\in {\mathbb N},$ $\rho_{k} \subseteq X \times Y$ is a binary relation and $\tilde{X}$ is a non-empty
subset of $X.$ If there exist an uncountable
set $S\subseteq \bigcap_{k=1}^{\infty} D(\rho_{k}) \cap \tilde{X}$ and
$\sigma>0$ such that for each $\epsilon>0$ and for each pair $x,\
y\in S$ of distinct points we have that for each $k\in {\mathbb N}$ there exist elements $x_{k}\in \rho_{k}x$ and $y_{k} \in \rho_{k}y$ such that \eqref{jednacina} holds, resp. \eqref{jednacinaj} [\eqref{jednacinaj1}/\eqref{jednacinaj2}] holds,
then we say that the sequence $(\rho_{k})_{k\in
{\mathbb N}}$ is
$\tilde{X}$-distributionally chaotic, resp. $\tilde{X}$-reiteratively distributionally chaotic [$\tilde{X}$-reiteratively distributionally chaotic of type $1$/$\tilde{X}$-reiteratively distributionally chaotic of type $2$].

The sequence $(\rho_{k})_{k\in
{\mathbb N}}$ is said to be densely\index{$\tilde{X}$-distributionally chaotic sequence of operators!densely}
$\tilde{X}$-distributionally chaotic, resp. $\tilde{X}$-reiteratively distributionally chaotic [$\tilde{X}$-reiteratively distributionally chaotic of type $1$/$\tilde{X}$-reiteratively distributionally chaotic of type $2$], iff $S$ can be chosen to be dense in $\tilde{X}.$
A binary relation $\rho \subseteq X \times X$ is said to be (densely)
$\tilde{X}$-distributionally chaotic, resp. $\tilde{X}$-reiteratively distributionally chaotic [$\tilde{X}$-reiteratively distributionally chaotic of type $1$/$\tilde{X}$-reiteratively distributionally chaotic of type $2$], iff the sequence $(\rho_{k}\equiv
\rho^{k})_{k\in {\mathbb N}}$ is.
The set $S$ is said to be $\sigma_{\tilde{X}}$-scrambled set, resp. $\sigma_{\tilde{X}}$-reiteratively scrambled set [$\sigma_{\tilde{X}}$-reiteratively scrambled set of type $1$/$\sigma_{\tilde{X}}$-reiteratively scrambled set of type $2$] ($\sigma$-scrambled set, resp. $\sigma$-reiteratively scrambled set [$\sigma$-reiteratively scrambled set of type $1$/$\sigma$-reiteratively scrambled set of type $2$], in the case that $\tilde{X}=X$)     
of the sequence $(\rho_{k})_{k\in {\mathbb N}}$ (the
binary relation $\rho$);  in the case that
$\tilde{X}=X,$ then we also say that the sequence $(\rho_{k})_{k\in
{\mathbb N}}$ (the binary relation $\rho$) is distributionally chaotic, resp. reiteratively distributionally chaotic [reiteratively distributionally chaotic of type $1$/reiteratively distributionally chaotic of type $2$].
\end{defn}

It is well known that,
for any infinite set $A\subseteq {\mathbb N}$, being syndetic and having a positive Banach density is the same thing. Therefore, if the sets $\{k \in {\mathbb N} :
d_{Y}(x_{k},y_{k})< \sigma \}$ and $\{k \in {\mathbb N} : d_{Y}(x_{k},y_{k})
\geq \epsilon \}$ are infinite, then they have unbounded difference sets iff \eqref{jednacinaj} holds. If one of these sets is finite, say the first one, then there exists $k_{0}=k_{0}(\sigma)\in {\mathbb N}$ such that $[k_{0},\infty) \subseteq \{k\in {\mathbb N} : d_{Y}(x_{k},y_{k})
\geq \epsilon\}$ and the second equality in \eqref{jednacinaj} cannot be satisfied. Therefore, in definition of $\tilde{X}$-reiterative distributional chaos, we can equivalently replace the equation \eqref{jednacinaj} with the statements that the difference sets of $\{k \in {\mathbb N} :
d_{Y}(x_{k},y_{k})< \sigma \}$ and $\{k \in {\mathbb N} : d_{Y}(x_{k},y_{k})
\geq \epsilon \}$ are unbounded.

The following definition seems to be new even for the sequences of linear not continuous operators on Banach and Fr\'echet spaces as well as for the sequences of linear continuous operators that are not orbits of one single operator:

\begin{defn}\label{Li-Yorke}
Suppose that, for every $k\in {\mathbb N},$ $\rho_{k} \subseteq X \times Y$ is a binary relation and $\tilde{X}$ is a non-empty subset of $X.$ If there exist an uncountable
set $S\subseteq \bigcap_{k=1}^{\infty} D(\rho_{k}) \cap \tilde{X}$ such that for each pair $x,\
y\in S$ of distinct points and for each $k\in {\mathbb N}$ there exist elements $x_{k}\in \rho_{k}x$ and $y_{k} \in \rho_{k}y$ such that 
\begin{align}\label{kxc}
\liminf_{k\rightarrow\infty}d_{Y}\bigl( x_{k},y_{k} \bigr)=0\mbox{  and  }\limsup_{k\rightarrow\infty}d_{Y}\bigl( x_{k},y_{k} \bigr)>0,
\end{align}
then we say that the sequence $(\rho_{k})_{k\in
{\mathbb N}}$ is
$\tilde{X}$-Li-Yorke chaotic.

The sequence $(\rho_{k})_{k\in
{\mathbb N}}$ is said to be densely\index{$\tilde{X}$-distributionally chaotic sequence of operators!densely}
$\tilde{X}$-Li-Yorke chaotic iff $S$ can be chosen to be dense in $\tilde{X}.$
A binary relation $\rho \subseteq X \times X$ is said to be (densely)
$\tilde{X}$-Li-Yorke chaotic iff the sequence $(\rho_{k}\equiv
\rho^{k})_{k\in {\mathbb N}}$ is.
The set $S$ is said to be $\tilde{X}$-scrambled Li-Yorke set (scrambled Li-Yorke set, in the case that $\tilde{X}=X$)     
of the sequence $(\rho_{k})_{k\in {\mathbb N}}$ (the
binary relation $\rho$);  in the case that
$\tilde{X}=X,$ then we also say that the sequence $(\rho_{k})_{k\in
{\mathbb N}}$ (the binary relation $\rho$) is (densely) Li-Yorke chaotic.
\end{defn}

Besides the conditions introduced so far, we can also examine the following ones:
\begin{align}\label{jednacinaq1}
\begin{split}
\underline{Bd}\Bigl( \bigl\{k \in {\mathbb N} :
d_{Y}\bigl(x_{k},y_{k}\bigr)< \sigma \bigr\}\Bigr)=0
\mbox{ and } \liminf_{k \rightarrow \infty} d_{Y}\bigl(x_{k},y_{k}\bigr)
=0,
\end{split}
\end{align}
\begin{align}\label{jednacinaq2}
\begin{split}
& \underline{d}\Bigl( \bigl\{k \in {\mathbb N} :
d_{Y}\bigl(x_{k},y_{k}\bigr)< \sigma \bigr\}\Bigr)=0\mbox{ and }\liminf_{k \rightarrow \infty} d_{Y}\bigl(x_{k},y_{k}\bigr)
=0,
\end{split}
\end{align}
\begin{align}\label{jednacinaq5}
\begin{split}
\limsup_{k\rightarrow\infty}d_{Y}\bigl( x_{k},y_{k} \bigr)>0
\mbox{ and }\underline{Bd}\Bigl( \bigl\{k \in {\mathbb N} :
d_{Y}\bigl(x_{k},y_{k}\bigr) \geq \epsilon \bigr\}\Bigr)=0,
\end{split}
\end{align}
\begin{align}\label{jednacinaq6}
\begin{split}
\limsup_{k\rightarrow\infty}d_{Y}\bigl( x_{k},y_{k} \bigr)>0
\mbox{ and }\underline{d}\Bigl( \bigl\{k \in {\mathbb N} :
d_{Y}\bigl(x_{k},y_{k}\bigr) \geq \epsilon \bigr\}\Bigr)=0.
\end{split}
\end{align}

The following definition is meaningful, as well:

\begin{defn}\label{DC-unbounded-fric-sorry} 
Suppose that, for every $k\in {\mathbb N},$ $\rho_{k} \subseteq X \times Y$ is a binary relation and $\tilde{X}$ is a non-empty
subset of $X.$ If there exist an uncountable
set $S\subseteq \bigcap_{k=1}^{\infty} D(\rho_{k}) \cap \tilde{X}$ and
$\sigma>0$ such that for each $\epsilon>0$ and for each pair $x,\
y\in S$ of distinct points we have that for each $k\in {\mathbb N}$ there exist elements $x_{k}\in \rho_{k}x$ and $y_{k} \in \rho_{k}y$ such that \eqref{jednacinaq1}/\eqref{jednacinaq6} holds,
then we say that the sequence $(\rho_{k})_{k\in
{\mathbb N}}$ is 
$(\tilde{X},1)$-mixed chaotic/$(\tilde{X},4)$-mixed chaotic.

The notion of densely $(\tilde{X},i)$-mixed chaotic sequence $(\rho_{k})_{k\in
{\mathbb N}}$ (the binary relation $\rho$), where $i\in {\mathbb N}_{4},$ the corresponding
$(\sigma_{\tilde{X}},i) $-mixed scrambled set
(($\sigma,i)$-mixed scrambled set, in the case that $\tilde{X}=X$), where $i\in {\mathbb N}_{2},$ the corresponding
$(\tilde{X},i) $-mixed scrambled set
($i$-mixed scrambled set, in the case that $\tilde{X}=X$), where $i\in {\mathbb N}_{4} \setminus {\mathbb N}_{2},$
of the sequence $(\rho_{k})_{k\in {\mathbb N}}$ (the
binary relation $\rho$) is introduced as above;  in the case that
$\tilde{X}=X$ and $i\in {\mathbb N}_{4},$ then we also say that the sequence $(\rho_{k})_{k\in
{\mathbb N}}$ (the binary relation $\rho$) is $i$-mixed chaotic. 
\end{defn}

Keeping in mind \eqref{jebu} and \eqref{jebu2}, an elementary line of reasoning shows that 
any $\tilde{X}$-distributionally chaotic sequence (binary relation) is already $\tilde{X}$-reiteratively distributionally chaotic as well as that any $\tilde{X}$-reiteratively distributionally chaotic sequence (binary relation) is both $\tilde{X}$-reiteratively distributionally chaotic of type $1$ and $\tilde{X}$-reiteratively distributionally chaotic of type $2$. 
It is also predictable that any $\tilde{X}$-reiteratively distributionally chaotic sequence of binary relations (any $\tilde{X}$-reiteratively distributionally chaotic binary relation) needs to be $\tilde{X}$-Li-Yorke chaotic. To see this, suppose that the sequence $(\rho_{k})_{k\in
{\mathbb N}}$ is $\tilde{X}$-reiteratively distributionally chaotic with given $\sigma_{\tilde{X}}$-scrambled set $S.$ Let $x=y-z$ for some two different elements $y,\, z\in S.$ Due to our assumption, for each number $k\in {\mathbb N}$ it is very simple to construct 
two strictly increasing sequences $(s_{n,k})_{n\in {\mathbb N}}$
and $(l_{n,k})_{n\in {\mathbb N}}$ of pairwise disjoint positive integers such that $d_{Y}(x_{s_{n,k}},0)\geq \sigma$
and $d_{Y}(x_{l_{n,k}},0)\leq 1/k$ for all $n\in {\mathbb N}$ as well as $l_{k+1,k+1}>l_{k,k}+2^{k},$ $n_{k+1,k+1}>n_{k,k}+2^{k}$
for all $k\in {\mathbb N}$ and $s_{1,1}<l_{1,1}<s_{2,2}<l_{2,2}< \cdot \cdot \cdot ;$ here, $x_{s_{n,k}}\in \rho_{s_{n,k}}x$
and $x_{l_{n,k}}\in \rho_{l_{n,k}}x.$ If $n\notin \bigcup_{k\in {\mathbb N}}\{s_{k,k},\, l_{k,k}\},$ then we take any vector $x_{n}\in \rho_{n}x,$ which clearly exists because the set  $\bigcap_{k=1}^{\infty} D(\rho_{k}) \cap \tilde{X}$
is at least countable. If $n=s_{k,k}$  ($n=l_{k,k}$) for some $k\in {\mathbb N},$ then we set $x_{n}:=x_{s_{k,k}}$ ($x_{n}:=x_{l_{k,k}}$). Then it is clear that $\liminf_{n\rightarrow\infty}d_{Y}( x_{n},0)=0$ and $\limsup_{n\rightarrow\infty}d_{Y}( x_{n},0)>0$ because $\lim_{n\rightarrow \infty}d_{Y}(x_{l_{n,n}},0)=0$ and the subsequence $(d_{Y}(x_{s_{n,n}},0))$ of $(d_{Y}(x_{n},0))$ is bounded away from zero. On the other hand, it is clear that $(\tilde{X},i)$-mixed chaos implies $\tilde{X}$-Li-Yorke chaos for any $i\in {\mathbb N}_{4};$ the
above conclusions also hold for any kind of dense $\tilde{X}$-chaos considered above.

Therefore, we have proved the following proposition:

\begin{prop}\label{funda}
Let for each $k\in {\mathbb N}$ we have that $\rho_{k}  \subseteq X \times Y$ is a binary relation, and let $\tilde{X}$ be 
a non-empty subset of $X.$ Consider the following statements:
\begin{itemize}
\item[(i)] $(\rho_{k})_{k\in
{\mathbb N}}$ is (densely) $\tilde{X}$-distributionally chaotic.
\item[(ii)] $(\rho_{k})_{k\in
{\mathbb N}}$ is (densely) reiteratively $\tilde{X}$-distributionally chaotic of type $1$.
\item[(iii)] $ (\rho_{k})_{k\in
{\mathbb N}}$ is (densely) reiteratively $\tilde{X}$-distributionally chaotic of type $2$.
\item[(iv)] $(\rho_{k})_{k\in
{\mathbb N}}$ is (densely) reiteratively $\tilde{X}$-distributionally chaotic.
\item[(v)] $(\rho_{k})_{k\in
{\mathbb N}}$ is (densely) $ (\tilde{X},1)$-mixed chaotic.
\item[(vi)] $(\rho_{k})_{k\in
{\mathbb N}}$ is (densely) $ (\tilde{X},2)$-mixed chaotic.
\item[(vii)] $(\rho_{k})_{k\in
{\mathbb N}}$ is (densely) $ (\tilde{X},3)$-mixed chaotic.
\item[(viii)] $(\rho_{k})_{k\in
{\mathbb N}}$ is (densely) $ (\tilde{X},4)$-mixed chaotic.
\item[(ix)] $(\rho_{k})_{k\in
{\mathbb N}}$ is (densely) 
$\tilde{X}$-Li-Yorke chaotic.
\end{itemize}
Then \emph{(i)} implies \emph{(ii)}-\emph{(ix)}; \emph{(ii)} 
implies \emph{(iv)}-\emph{(v)}, \emph{(vii)}-\emph{(viii)}  and \emph{(ix)}; \emph{(iii)} implies \emph{(iv)}-\emph{(vii)}  and \emph{(ix)};
\emph{(iv)} implies \emph{(v)}, \emph{(vii)} and \emph{(ix)}; \emph{(v)}, \emph{(vii)} or \emph{(viii)} implies \emph{(ix)};
\emph{(vi)} implies \emph{(v)} and \emph{(ix)}.
\end{prop}

It is worth noting that any two different types of chaos considered above, as well as in Definition \ref{DC-unbounded-frick} and Definition \ref{DC-unbounded-fric-sorry} below, do not coincide even for the sequences of continuous linear operators on finite-dimensional spaces. This can be inspected as in Example \ref{primeran} below.

\section{Distributional chaos of type $s$ for binary relations ($s\in \{1,2,2\frac{1}{2},3\}$)}\label{marek-MLOsk}

As in the previous one, in this section we assume that $(X,d)$ and  $(Y,d_{Y})$ are metric spaces.
Suppose that $\sigma,\ \sigma'>0$, $\epsilon>0$ and $(x_{k})_{k\in {\mathbb N}},\ (y_{k})_{k\in {\mathbb N}}$ are two given sequences in $Y.$
Consider the following conditions:
\begin{align}\label{jednacinak11}
\begin{split}
& \overline{Bd}\Bigl( \bigl\{k \in {\mathbb N} :
d_{Y}\bigl(x_{k},y_{k}\bigr)\geq \sigma \bigr\}\Bigr)>0,
\\
& \underline{Bd}\Bigl(\bigl\{k \in {\mathbb N} : d_{Y}\bigl(x_{k},y_{k}\bigr)
\geq \epsilon \bigr\}\Bigr)=0;
\end{split}
\end{align}
\begin{align}\label{jednacinak12}
\begin{split}
& \overline{Bd}\Bigl( \bigl\{k \in {\mathbb N} :
d_{Y}\bigl(x_{k},y_{k}\bigr)\geq \sigma \bigr\}\Bigr)>0,
\\
& \underline{d}\Bigl(\bigl\{k \in {\mathbb N} : d_{Y}\bigl(x_{k},y_{k}\bigr)
\geq \epsilon \bigr\}\Bigr)=0;
\end{split}
\end{align}
\begin{align}\label{jednacinak13}
\begin{split}
& \overline{d}\Bigl( \bigl\{k \in {\mathbb N} :
d_{Y}\bigl(x_{k},y_{k}\bigr)\geq \sigma \bigr\}\Bigr)>0,
\\
& \underline{Bd}\Bigl(\bigl\{k \in {\mathbb N} : d_{Y}\bigl(x_{k},y_{k}\bigr)
\geq \epsilon \bigr\}\Bigr)=0;
\end{split}
\end{align}
\begin{align}\label{jednacinak14}
\begin{split}
& \overline{d}\Bigl( \bigl\{k \in {\mathbb N} :
d_{Y}\bigl(x_{k},y_{k}\bigr)\geq \sigma \bigr\}\Bigr)>0,
\\
& \underline{d}\Bigl(\bigl\{k \in {\mathbb N} : d_{Y}\bigl(x_{k},y_{k}\bigr)
\geq \epsilon \bigr\}\Bigr)=0;
\end{split}
\end{align}
there exist $c >0$ and $r>0$ such that 
\begin{align}\label{jednacinak21}
\underline{Bd}\Bigl( \bigl\{k \in {\mathbb N} :
d_{Y}\bigl(x_{k},y_{k}\bigr)< \sigma  \bigr\}\Bigr)<c<\overline{Bd}\Bigl( \bigl\{k \in {\mathbb N} :
d_{Y}\bigl(x_{k},y_{k}\bigr)< \sigma\bigr\}\Bigr)
\end{align}
for $0<\sigma<r;$\\
there exist $c >0$ and $r>0$ such that 
\begin{align}\label{jednacinak22}
\underline{Bd}\Bigl( \bigl\{k \in {\mathbb N} :
d_{Y}\bigl(x_{k},y_{k}\bigr)< \sigma  \bigr\}\Bigr)<c<\overline{d}\Bigl( \bigl\{k \in {\mathbb N} :
d_{Y}\bigl(x_{k},y_{k}\bigr)< \sigma\bigr\}\Bigr)
\end{align}
for $0<\sigma<r;$\\
there exist $c >0$ and $r>0$ such that 
\begin{align}\label{jednacinak23}
\underline{d}\Bigl( \bigl\{k \in {\mathbb N} :
d_{Y}\bigl(x_{k},y_{k}\bigr)< \sigma  \bigr\}\Bigr)<c<\overline{Bd}\Bigl( \bigl\{k \in {\mathbb N} :
d_{Y}\bigl(x_{k},y_{k}\bigr)< \sigma\bigr\}\Bigr)
\end{align}
for $0<\sigma<r;$\\
there exist $c >0$ and $r>0$ such that 
\begin{align}\label{jednacinak24}
\underline{d}\Bigl( \bigl\{k \in {\mathbb N} :
d_{Y}\bigl(x_{k},y_{k}\bigr)< \sigma  \bigr\}\Bigr)<c<\overline{d}\Bigl( \bigl\{k \in {\mathbb N} :
d_{Y}\bigl(x_{k},y_{k}\bigr)< \sigma\bigr\}\Bigr)
\end{align}
for $0<\sigma<r;$
\begin{align}\label{jednacinak31}
\mbox{there exist }a,\ b,\ c >0 \mbox{ such that }\eqref{jednacinak21} \mbox{ holds for all }\sigma \in [a,b];
\end{align}
\begin{align}\label{jednacinak32}
\mbox{there exist }a,\ b,\ c >0 \mbox{ such that }\eqref{jednacinak22} \mbox{ holds for all }\sigma \in [a,b];
\end{align}
\begin{align}\label{jednacinak33}
\mbox{there exist }a,\ b,\ c >0 \mbox{ such that }\eqref{jednacinak23} \mbox{ holds for all }\sigma \in [a,b];
\end{align}
\begin{align}\label{jednacinak34}
\mbox{there exist }a,\ b,\ c >0 \mbox{ such that }\eqref{jednacinak24} \mbox{ holds for all }\sigma \in [a,b].
\end{align}

Let $i\in \{1,2\}.$ For (reiterative) distributional chaos (reiterative distributional chaos of type $i$) it is also said that it is (reiterative) distributional chaos of type $0;1$ (reiterative distributional chaos of type $i;1$). 
Now we would like to propose the following notion:

\begin{defn}\label{DC-unbounded-frick} 
Let $i\in \{0,1,2\}$ and $s\in \{1,2,2\frac{1}{2},3\}.$
Suppose that, for every $k\in {\mathbb N},$ $\rho_{k} \subseteq X \times Y$ is a binary relation and $\tilde{X}$ is a non-empty subset of $X.$ 
\begin{itemize}
\item[(i)] If there exist an uncountable
set $S\subseteq \bigcap_{k=1}^{\infty} D(\rho_{k}) \cap \tilde{X}$ and
$\sigma>0$ such that for each $\epsilon>0$ and for each pair $x,\
y\in S$ of distinct points we have that for each $k\in {\mathbb N}$ there exist elements $x_{k}\in \rho_{k}x$ and $y_{k} \in \rho_{k}y$ such that \eqref{jednacinak11}
[\eqref{jednacinak12}/\eqref{jednacinak13}/\eqref{jednacinak14}] holds, then we say that the sequence $(\rho_{k})_{k\in
{\mathbb N}}$ is reiteratively
$\tilde{X}$-distributionally chaotic of type $0;2$ [reiteratively
$\tilde{X}$-distributionally chaotic of type $1;2$/reiteratively
$\tilde{X}$-distributionally chaotic of type $2;2$/$\tilde{X}$-distributionally chaotic of type $0;2$].
\item[(ii)] If there exist an uncountable
set $S\subseteq \bigcap_{k=1}^{\infty} D(\rho_{k}) \cap \tilde{X}$ and numbers
$\sigma,\ c, \  r>0$ such that for each $\epsilon>0$ and for each pair $x,\
y\in S$ of distinct points we have that for each $k\in {\mathbb N}$ there exist elements $x_{k}\in \rho_{k}x$ and $y_{k} \in \rho_{k}y$ such that \eqref{jednacinak21}
[\eqref{jednacinak22}/\eqref{jednacinak23}/\eqref{jednacinak24}] holds for $0<\sigma<r$, then we say that the sequence $(\rho_{k})_{k\in
{\mathbb N}}$ is reiteratively $\tilde{X}$-distributionally chaotic of type $2\frac{1}{2}$ [reiteratively
$\tilde{X}$-distributionally chaotic of type $1;2\frac{1}{2}$/reiteratively
$\tilde{X}$-distributionally chaotic of type $2;2\frac{1}{2}$/$\tilde{X}$-distributionally chaotic of type $2\frac{1}{2}$].
\item[(iii)] If there exist an uncountable
set $S\subseteq \bigcap_{k=1}^{\infty} D(\rho_{k}) \cap \tilde{X}$ and real numbers
$\sigma,\ a,\ b,\ c>0$ such that for each $\epsilon>0$ and for each pair $x,\
y\in S$ of distinct points we have that for each $k\in {\mathbb N}$ there exist elements $x_{k}\in \rho_{k}x$ and $y_{k} \in \rho_{k}y$ such that \eqref{jednacinak31}
[\eqref{jednacinak32}/\eqref{jednacinak33}/\eqref{jednacinak34}] holds for $a<\sigma<b$, then we say that the sequence $(\rho_{k})_{k\in
{\mathbb N}}$ is reiteratively $\tilde{X}$-distributionally chaotic of type $3$ [reiteratively
$\tilde{X}$-distributionally chaotic of type $1;3$/reiteratively
$\tilde{X}$-distributionally chaotic of type $2;3$/$\tilde{X}$-distributionally chaotic of type $3$].
\end{itemize}
The sequence $(\rho_{k})_{k\in
{\mathbb N}}$ is said to be densely (reiteratively)
$\tilde{X}$-distributionally chaotic of type $i;s$ iff $S$ can be chosen to be dense in $\tilde{X}.$
A binary relation $\rho \subseteq X \times X$ is said to be (densely) reiteratively
$\tilde{X}$-distributionally chaotic of type $i;s$ iff
the sequence $(\rho_{k}\equiv
\rho^{k})_{k\in {\mathbb N}}$ is.
The set $S$ is said to be (reiteratively)
$(\sigma_{\tilde{X}},s)$-scrambled set ((reiteratively) $(\sigma,s)$-scrambled set,  in the case that $\tilde{X}=X$)     
of the sequence $(\rho_{k})_{k\in {\mathbb N}}$ (the
binary relation $\rho$);  in the case that
$\tilde{X}=X,$ then we also say that the sequence $(\rho_{k})_{k\in
{\mathbb N}}$ (the binary relation $\rho$) is  densely (reiteratively)
distributionally chaotic of type $i;s$.
\end{defn}

Keeping in mind the inequality \eqref{jebu}, we are in a position to immediately clarify the following:

\begin{prop}\label{mace}
Suppose that $i\in \{0,1,2\},$ $s,\ s_{1},\ s_{2} \in \{1,2,2\frac{1}{2},3\},$ $s_{1}\leq s_{2},$ $\tilde{X}$ is a non-empty subset of $X$ and $(\rho_{k})_{k\in
{\mathbb N}}$ is a given sequence of binary relations. Then, for $(\rho_{k})_{k\in
{\mathbb N}},$ we have the following:
\begin{center}
(dense, reiterative) 
$\tilde{X}$-distributional chaos of type $i;s_{1}$ implies (dense, reiterative) $\tilde{X}$-distributional chaos of type $i;s_{2}$,
\end{center}
\begin{center}
(dense)
$\tilde{X}$-distributional chaos of type $0;s$ implies (dense) reiterative $\tilde{X}$-distributional chaos of types $1;s$ and $2;s$ 
\end{center}
and
\begin{center}
(dense) reiterative $\tilde{X}$-distributional chaos of type $1;s$ or  $2;s$ implies (dense) reiterative $\tilde{X}$-distributional chaos of type $0;s.$ 
\end{center}
\end{prop}

As it is well known, $\tilde{X}$-distributional chaos of type $3$ is a very weak form of linear chaos: it is still unknown whether there exists a complex matrix that is distributionally chaotic of type $3$ (cf. \cite[Problem 51]{2018JMMA}). The same problem can be posed for reiterative distributional chaos of type $3.$ 

Concerning the relation of distributional chaos
and distributional chaos of type $2,$ it is worth noting that these two notions are equivalent for linear continuous operators on Banach spaces (see \cite[Theorem 2]{2018JMMA}). As mentioned at the end of previous section, this is far from being true for general sequences of linear continuous operators. 

A further analysis of distributional chaos of type $s$ and their generalizations will be carried out somewhere else.

\section{Distributional chaos and Li-Yorke chaos  in Fr\'echet spaces}\label{frechet-mlos}

In the remaining part of the paper, we assume that $X$ is an
infinite-dimensional Fr\' echet space over the field ${\mathbb K}\in \{{\mathbb R},\, {\mathbb C}\}$ and that the topology of $X$ is
induced by the fundamental system $(p_{n})_{n\in {\mathbb N}}$ of
increasing seminorms (separability of $X$ is not assumed a priori in future). Then the translation invariant metric\index{ translation invariant metric} $d :
X\times X \rightarrow [0,\infty),$ defined by
\begin{equation}\label{metri}
d(x,y):=\sum
\limits_{n=1}^{\infty}\frac{1}{2^{n}}\frac{p_{n}(x-y)}{1+p_{n}(x-y)},\quad
x,\ y\in X,
\end{equation}
enjoys the following properties:
$$
d(x+u,y+v)\leq d(x,y)+d(u,v),\quad x,\ y,\ u,\
v\in X,
$$
\begin{align*}
d(cx,cy)\leq (|c|+1)d(x,y),\ c\in {\mathbb K},\quad x,\ y\in X,
\end{align*}
and
\begin{align*}
d(\alpha x,\beta x)\geq \frac{|\alpha-\beta|}{1+|\alpha-\beta|}d(0,x),\quad x\in X,\ \alpha,\ \beta \in
{\mathbb K}.
\end{align*}
By $Y$ we denote another
Fr\' echet space over the same field of scalars as $X;$ the topology of $Y$ will be induced by the
fundamental system $(p_{n}^{Y})_{n\in {\mathbb N}}$ of increasing
seminorms.
Define the translation invariant metric $d_{Y} :
Y\times Y \rightarrow [0,\infty)$ by replacing $p_{n}(\cdot)$ with
$p_{n}^{Y}(\cdot)$ in (\ref{metri}).
If
$(X,\|\cdot \|)$ or $(Y,\|\cdot \|_{Y})$ is a Banach space, then we assume that the
distance of two elements $x,\ y\in X$ ($x,\ y\in Y$) is given by $d(x,y):=\|x-y\|$ ($d_{Y}(x,y):=\|x-y\|_{Y}$).
Keeping in mind this terminological change,
our structural results clarified in Fr\' echet spaces continue to hold in the case that $X$ or $Y$ is a Banach space. By $L(X,Y)$ we denote the space consisting of all linear continuous mappings from $X$ into $Y;$ $L(X)\equiv L(X,X).$

In Fr\'echet spaces, of importance are the following conditions:

\begin{align}\label{jednacinaq32}
\mbox{ the sequence }\bigl( x_{k}-y_{k} \bigr)_{k\in {\mathbb N}} \mbox{ is unbounded
and }\liminf_{k\rightarrow\infty}d_{Y}\bigl( x_{k},y_{k} \bigr)=0,
\end{align}
\begin{align}\label{jednacinaq33}
\begin{split}
\mbox{ the sequence }\bigl( x_{k}-y_{k} \bigr)_{k\in {\mathbb N}} \mbox{ is unbounded
and }\underline{Bd}\Bigl( \bigl\{k \in {\mathbb N} :
d_{Y}\bigl(x_{k},y_{k}\bigr) \geq \epsilon \bigr\}\Bigr)=0,
\end{split}
\end{align}
\begin{align}\label{jednacinaq34}
\begin{split}
\mbox{ the sequence }\bigl( x_{k}-y_{k} \bigr)_{k\in {\mathbb N}} \mbox{ is unbounded
and }\underline{d}\Bigl( \bigl\{k \in {\mathbb N} :
d_{Y}\bigl(x_{k},y_{k}\bigr) \geq \epsilon \bigr\}\Bigr)=0.
\end{split}
\end{align}

Albeit the most intriguing for multivalued linear operators, the following notion can be introduced for general binary relations, as well:

\begin{defn}\label{DC-unbounded-fric-sorry} 
Suppose that, for every $k\in {\mathbb N},$ $\rho_{k} \subseteq X \times Y$ is a binary relation and $\tilde{X}$ is a non-empty
subset of $X.$ If there exists an uncountable
set $S\subseteq \bigcap_{k=1}^{\infty} D(\rho_{k}) \cap \tilde{X}$ such that for each pair $x,\
y\in S$ of distinct points and for each $\epsilon>0$ we have that for each $k\in {\mathbb N}$ there exist elements $x_{k}\in \rho_{k}x$ and $y_{k} \in \rho_{k}y$ such that \eqref{jednacinaq32} holds, resp. \eqref{jednacinaq33} [\eqref{jednacinaq34}] holds,
then we say that the sequence $(\rho_{k})_{k\in
{\mathbb N}}$ is strongly $\tilde{X}$-Li-Yorke chaotic, resp. 
$\langle \tilde{X},1 \rangle$-mixed chaotic [$\langle \tilde{X},2 \rangle$-mixed chaotic].

The notion of densely strong $\tilde{X}$-Li-Yorke chaotic sequence, resp. densely $\langle \tilde{X},i \rangle$-mixed chaotic sequence $(\rho_{k})_{k\in
{\mathbb N}}$ (the binary relation $\rho$), where $i\in {\mathbb N}_{2},$ the corresponding strong $\tilde{X}$-Li-Yorke scrambled set, resp.
$\langle \tilde{X},i\rangle $-mixed scrambled set
(strong Li-Yorke scrambled set, resp. $\langle i \rangle$-mixed scrambled set, in the case that $\tilde{X}=X$), where $i\in {\mathbb N}_{2},$
of the sequence $(\rho_{k})_{k\in {\mathbb N}}$ (the
binary relation $\rho$) is introduced as above;  in the case that
$\tilde{X}=X$ and $i\in {\mathbb N}_{2},$ then we also say that the sequence $(\rho_{k})_{k\in
{\mathbb N}}$ (the binary relation $\rho$) is strong Li-Yorke chaotic, resp. $i$-mixed chaotic. 
\end{defn}

With the exception of implications clarified in Proposition \ref{funda}, we can only state the following ones, in general:
\begin{itemize}
\item[(A)] strong $\tilde{X}$-Li-Yorke chaos implies $\tilde{X}$-Li-Yorke chaos;
\item[(B)] $\langle \tilde{X},1 \rangle$-mixed chaos implies strong $\tilde{X}$-Li-Yorke chaos;
\item[(C)]  $\langle \tilde{X},2 \rangle$-mixed chaos implies  $\langle \tilde{X},1 \rangle$-mixed chaos and strong $\tilde{X}$-Li-Yorke chaos.
\end{itemize}

We continue by observing the following:
If $X$ is a Banach space and $T\in L(X),$ then $T$ is (densely) Li-Yorke chaotic iff $T$ is (densely) reiteratively distributionally chaotic.
To see this, let us recall that we have  $\|T\|>1$ due to the fact that
$T$ is Li-Yorke chaotic. 
By \cite[Theorem 5]{2011}, we have the existence of a vector $x\in X$ such that 
$\liminf_{n\rightarrow\infty}\| T^{n}x\|=0$ and $\limsup_{n\rightarrow\infty}\| T^{n}x\|=\infty.$ Therefore, there exist two strictly increasing sequences of positive integers $(n_{k})$ and $(l_{k})$ such that $\min(n_{k+1}-n_{k},l_{k+1}-l_{k})>k^{2},$ $\| T^{n_{k}}x\|<2^{-k^{2}}$ and $\| T^{l_{k}}x\|>2^{k^{2}}$ for all $k\in {\mathbb N}.$ Put $A:=\bigcup_{k\in {\mathbb N}}[n_{k},n_{k}+k]$ and
$B:=\bigcup_{k\in {\mathbb N}}[l_{k},l_{k}-k].$ Then $\overline{Bd}(A)=\overline{Bd}(B)=1$ and for each $n\in A$ ($n\in B$) there exists 
$k\in {\mathbb N}$ such that $n\in [n_{k},n_{k}+k]$ ($n\in [l_{k},l_{k}-k]$) and therefore $\|T^{n}x\|\leq (1+\|T\|)^{k}2^{-k^{2}}$
($\|T^{n}x\|\geq \|T\|^{-k}2^{k^{2}}$). This, in turn, implies that the notions of $1$-mixed chaos, $3$-mixed chaos, reiterative distributional chaos and Li-Yorke chaos coincide in this case (this also holds for dense analogues). 

On the other hand, as already mentioned, the situation is completely different for the sequences of continuous linear operators on finite-dimensional spaces. Take, for instance, $T_{k}=0$ if $k$ is even and $T_{k}=I$ if $k$ is odd. Then the sequence $(T_{k})$ is Li-Yorke chaotic on ony Fr\' echet space $X$ but it is not strongly Li-Yorke chaotic (this trivial counterexample also shows that the assertions of \cite[Theorem 5]{2011} and \cite[Theorem 9]{band}, where it has been proved that the notions of (dense) Li-Yorke chaos and (dense) strong Li-Yorke chaos coincide for the orbits of linear continuous operators, do not hold for the sequences of continuous linear operators on Banach and Fr\'echet function spaces). 

If $({\rho}_{k})_{k\in {\mathbb N}}$ and $\tilde{X}$ are given in advance, then we define the binary relations
${\mathbb \rho }_{k}' : D({\mathbb \rho }_{k}')\subseteq X \rightarrow Y$ by $D({\mathbb \rho }_{k}'):=D({\rho}_{k}) \cap \tilde{X}$ and
${\mathbb \rho }_{k}'x:={\rho}_{k}x,$ $x\in D({\mathbb \rho}_{k}')$ ($k\in {\mathbb N}$). We can simply prove the following proposition:

\begin{prop}\label{subspace-fric}
Let $i\in \{1,2\}.$ Then 
$({\rho}_{k})_{k\in {\mathbb N}}$ is (reiteratively) $\tilde{X}$-distributionally chaotic, resp. reiteratively $\tilde{X}$-distributionally chaotic of type $i$/(strong) $\tilde{X}$-Li-Yorke chaotic,
iff $({\mathbb \rho}_{k}')_{k\in {\mathbb N}}$ is (reiteratively) distributionally chaotic, resp. reiteratively distributionally chaotic of type $i$/(strong) Li-Yorke chaotic. The same holds for $\langle \tilde{X},i \rangle$-mixed chaos and $(\tilde{X},j)$-mixed chaos, where $j\in {\mathbb N}_{4}.$
\end{prop}

In our further work, we will consider only the sequences $({\mathcal A}_{k})_{k\in
{\mathbb N}}$ of MLOs between the spaces $X$ and $Y$ as well as the orbits of an MLO ${\mathcal A}$ in $X.$ First of all, we would like to observe
the following:

\begin{rem}\label{besko}
Let $S:=\bigcap_{k=1}^{\infty} D({\mathcal A}_{k}) \neq \{0\}$ and let any operator ${\mathcal A}_{k}$ be purely multivalued ($k\in {\mathbb N}$). Choosing numbers
$\sigma>0,\ \epsilon>0$ and each pair $x,\ y\in S$ of distinct points arbitrarily, we can always find appropriate elements $x_{ k}\in {\mathcal A}_{k}x$ and $y_{k}\in {\mathcal A}_{k}x$ such that the set
$\{k \in {\mathbb N} :
d_{Y}(x_{k},y_{k})< \sigma \}$ is finite and
the first equations in
\eqref{jednacina} and \eqref{jednacinaj} automatically hold. Therefore, it is very important to assume that the second parts in the equations, e.g. \eqref{jednacina} and \eqref{jednacinaj}-\eqref{jednacinaj2}, hold with the same elements $x_{k}\in {\mathcal A}_{k}x$ and $y_{k}\in {\mathcal A}_{k}x$ (not for some other elements $x_{k}'\in {\mathcal A}_{k}x$ and $y_{k}'\in {\mathcal A}_{k}x$). If we accept this weaker notion of distributional chaos and reiterative distributional chaos (of type $1$ or $2$), with different vectors $x_{k}'\in {\mathcal A}_{k}x$ and $y_{k}'\in {\mathcal A}_{k}x$ in the second equality of \eqref{jednacina} and \eqref{jednacinaj}-\eqref{jednacinaj2}, then 
we will be in a position to construct a great number of densely distributionally chaotic operators and sequences of MLOs. 
For example,
suppose that $A\in L(X)$ and the linear subspace $X_{0}:=\{x\in X : \lim_{k\rightarrow \infty}A^{k}x=0\}$ is dense in $X.$ Set ${\mathcal A}_{k}x:=A^{k}x+W_{k},$ $k\in {\mathbb N},$ where $W_{k}\neq \{0\}$ is a subspace of $X$ ($k\in {\mathbb N}$). Since the first equation in
\eqref{jednacina} holds, setting $S:=X_{0}$ and $x_{k}':=A^{k}x,$  $y_{k}':=A^{k}y$ ($k\in {\mathbb N},$ $x,\ y\in X_{0}$), it readily follows that the sequence $({\mathcal A}_{k})_{k\in {\mathbb N}}$ will be densely distributionally chaotic in this weaker sense.
In the sequel, we will follow solely the notion in which $x_{k}'=x_{k}$ and $y_{k}'=y_{k}$ ($k\in {\mathbb N}$).
\end{rem}

We can simply verify that the notions of distributional chaos, reiterative distributional chaos, reiterative distributional chaos of type $1$ and reiterative distributional chaos of type $2$ do not coincide: 

\begin{example}\label{primeran}
It is well known that a subset $A$ of ${\mathbb N}$ has the upper Banach density $1$ iff,
for every integer $d\in {\mathbb N},$ the set $A$ contains infinitely many pairwise disjoint intervals of $d$ consecutive
integers. Therefore, it is very 
simple to construct two disjoint subsets $A$ and $B$ of ${\mathbb N}$ such that ${\mathbb N}=A\cup B,$ $\overline{d}(A)<1$ and $\overline{Bd}(A)=\overline{Bd}(B)=1.$ After that, set 
$X:={\mathbb K}$, $T_{k}:=kI$ ($k\in A$) and $T_{k}:=0$ ($k\in B$). Then it can be simply checked that the sequence $(T_{k})_{k\in {\mathbb N}}$
is reiteratively distributionally chaotic but not reiteratively distributionally chaotic of type $2$, as well as that the corresponding reiteratively scrambled set $S$ can be chosen to be the whole space $X.$ Furthermore, there exist two possible subcases: $\overline{d}(B)=1$ or $\overline{d}(B)<1.$ In the first one, the sequence $(T_{k})_{k\in {\mathbb N}}$ is reiteratively distributionally chaotic of type $1$, while in the second one the sequence $(T_{k})_{k\in {\mathbb N}}$ is not reiteratively distributionally chaotic of type $1.$ Keeping in mind the obvious symmetry between the reiterative distributional chaos of type $1$ and reiterative distributional chaos of type $2,$ we obtain the claimed.  
\end{example}

\subsection{Irregular vectors and irregular manifolds}\label{irregular}

We start this section by introducing the following notion (cf. \cite[Definition 18]{2013JFA} and \cite[Definition 3.4]{mendoza} for single-valued linear case): 

\begin{defn}\label{DC-unbounded-fric-prim}
Suppose that for each $k\in {\mathbb N},$ ${\mathcal A}_{k} : D({\mathcal A}_{k})\subseteq X \rightarrow Y$ is an MLO, $\tilde{X}$ is a closed linear
subspace of $X,$ $x\in \bigcap_{k=1}^{\infty}D({\mathcal A}_{k})$ and
$m\in {\mathbb N}.$ Then we say that:
\begin{itemize}
\item[(i)] $x$ is (reiteratively) distributionally
near to $0$ for $({\mathcal A}_{k})_{k\in {\mathbb N}}$
iff there exists $A\subseteq {\mathbb N}$ such that ($\overline{Bd}(A)=1$) $\overline{d}(A)=1$
and for each $k\in A$ there exists $x_{k}\in {\mathcal A}_{k}x$ such that
$\lim_{k\in A,k\rightarrow \infty}x_{k}=0;$
\item[(ii)] $x$ is (reiteratively) distributionally $m$-unbounded for $({\mathcal A}_{k})_{k\in {\mathbb N}}$ iff there exists $B\subseteq
{\mathbb N}$ such that ($\overline{Bd}(B)=1$) $\overline{d}(B)=1$ and for each
$k\in B$ there exists $x_{k}'\in {\mathcal A}_{k}x$ such that
$\lim_{k\in
B,k\rightarrow \infty}p_{m}^{Y}(x_{k}')=\infty ;$ $x$ is said to be (reiteratively) distributionally unbounded for $({\mathcal A}_{k})_{k\in {\mathbb N}}$ iff there
exists $q\in {\mathbb N}$ such that $x$ is (reiteratively) distributionally $q$-unbounded for $({\mathcal A}_{k})_{k\in {\mathbb N}}$
(if $Y$ is a Banach space, this simply means that  $\lim_{k\in
B,k\rightarrow \infty}\|x_{k}'\|_{Y}=\infty );$
\item[(iii)] $x$ is a (reiteratively) $\tilde{X}$-distributionally irregular vector\index{$\tilde{X}$-distributionally irregular vector} for
$({\mathcal A}_{k})_{k\in {\mathbb N}}$ iff $x\in
\bigcap_{k=1}^{\infty}D({\mathcal A}_{k}) \cap \tilde{X},$ (i) holds with with some subset $A$ of ${\mathbb N}$ satisfying $\overline{d}(A)=1$ ($\overline{Bd}(A)=1$) and 
the sequence $(x_{k})$, as well as the second part of
(ii) holds with some subset $B$ of ${\mathbb N}$ satisfying $\overline{d}(B)=1$ ($\overline{Bd}(B)=1$)  
and the same sequence $(x_{k}'=x_{k})$ as in (i) ({\it for the sake of brevity, we will assume in any part \emph{(iv)}-\emph{(xiii)} below that $x_{k}'=x_{k},$ with the meaning clear});
\item[(iv)] $x$ is a reiteratively $\tilde{X}$-distributionally irregular vector of type $1$ for
$({\mathcal A}_{k})_{k\in {\mathbb N}}$ iff $x\in \tilde{X}$ is distributionally near to zero and $x$ is reiterativelty distributionally chaotic for $({\mathcal A}_{k})_{k\in {\mathbb N}};$
\item[(v)] $x$ is a reiteratively $\tilde{X}$-distributionally irregular vector of type $2$ for
$({\mathcal A}_{k})_{k\in {\mathbb N}}$ iff $x\in \tilde{X}$ is reiteratively distributionally near to zero and $x$ is distributionally chaotic for $({\mathcal A}_{k})_{k\in {\mathbb N}};$
\item[(vi)] $x$ is a strong $\tilde{X}$-Li-Yorke irregular vector
for $({\mathcal A}_{k})_{k\in {\mathbb N}}$
iff  $x\in
\bigcap_{k=1}^{\infty}D({\mathcal A}_{k}) \cap \tilde{X}$ and for each $k\in {\mathbb N}$ there exists $x_{k}\in {\mathcal A}_{k}x$ such that $(x_{k})_{k\in {\mathbb N}}$ is unbounded and has a subsequence converging to zero;
\item[(vii)] $x$ is a $\tilde{X}$-Li-Yorke irregular vector
for $({\mathcal A}_{k})_{k\in {\mathbb N}}$
iff  $x\in
\bigcap_{k=1}^{\infty}D({\mathcal A}_{k}) \cap \tilde{X}$ and for each $k\in {\mathbb N}$ there exists $x_{k}\in {\mathcal A}_{k}x$ such that  $(x_{k})_{k\in {\mathbb N}}$
does not converge to zero but it has a subsequence converging to zero;
\item[(viii)] $x$ is a $(\tilde{X},1)$-distributionally irregular vector
for $({\mathcal A}_{k})_{k\in {\mathbb N}}$
iff $x$ is reiteratively distributionally unbounded for $({\mathcal A}_{k})_{k\in {\mathbb N}}$ and for each $k\in {\mathbb N}$ there exists $x_{k}\in {\mathcal A}_{k}x$ such that  $(x_{k})_{k\in {\mathbb N}}$
has a subsequence converging to zero;
\item[(ix)] $x$ is a $(\tilde{X},2)$-distributionally irregular vector
for $({\mathcal A}_{k})_{k\in {\mathbb N}}$
iff $x\in \tilde{X}$ is distributionally unbounded for $({\mathcal A}_{k})_{k\in {\mathbb N}}$ and for each $k\in {\mathbb N}$ there exists $x_{k}\in {\mathcal A}_{k}x$ such that  $(x_{k})_{k\in {\mathbb N}}$
has a subsequence converging to zero;
\item[(x)] $x$ is 
a $(\tilde{X},3)$-distributionally irregular vector
for $({\mathcal A}_{k})_{k\in {\mathbb N}}$ iff 
for each $k\in {\mathbb N}$ there exists $x_{k}\in {\mathcal A}_{k}x$ such that $(x_{k})_{k\in {\mathbb N}}$
does not converge to zero and $x\in \tilde{X}$ is 
reiteratively distributionally
near to $0$ for $({\mathcal A}_{k})_{k\in {\mathbb N}};$
\item[(xi)] $x$ is 
a $(\tilde{X},4)$-distributionally irregular vector
for $({\mathcal A}_{k})_{k\in {\mathbb N}}$ iff 
for each $k\in {\mathbb N}$ there exists $x_{k}\in {\mathcal A}_{k}x$ such that $(x_{k})_{k\in {\mathbb N}}$
does not converge to zero and $x\in \tilde{X}$ is 
distributionally
near to $0$ for $({\mathcal A}_{k})_{k\in {\mathbb N}};$
\item[(xii)] $x$ is a $\langle \tilde{X},1 \rangle$-mixed chaotic irregular vector
for $({\mathcal A}_{k})_{k\in {\mathbb N}}$
iff  $x\in
\bigcap_{k=1}^{\infty}D({\mathcal A}_{k}) \cap \tilde{X}$ and for each $k\in {\mathbb N}$ there exists $x_{k}\in {\mathcal A}_{k}x$ such that  $(x_{k})_{k\in {\mathbb N}}$
is unbounded and $x$ is reiteratively distributionally
near to $0$ for
$({\mathcal A}_{k})_{k\in {\mathbb N}}$;
\item[(xiii)] $x$ is a $\langle \tilde{X},2 \rangle$-mixed chaotic irregular vector
for $({\mathcal A}_{k})_{k\in {\mathbb N}}$
iff  $x\in
\bigcap_{k=1}^{\infty}D({\mathcal A}_{k}) \cap \tilde{X}$ and for each $k\in {\mathbb N}$ there exists $x_{k}\in {\mathcal A}_{k}x$ such that  $(x_{k})_{k\in {\mathbb N}}$
is unbounded and $x$ is distributionally
near to $0$ for
$({\mathcal A}_{k})_{k\in {\mathbb N}};$
\item[(xiv)] $x$ is $\tilde{X}$-Li-Yorke near to zero 
for $({\mathcal A}_{k})_{k\in {\mathbb N}}$
iff  $x\in
\bigcap_{k=1}^{\infty}D({\mathcal A}_{k}) \cap \tilde{X}$ and for each $k\in {\mathbb N}$ there exists $x_{k}\in {\mathcal A}_{k}x$ such that $(x_{k})_{k\in {\mathbb N}}$ has a  subsequence converging to zero.
\end{itemize}
If ${\mathcal A} : D({\mathcal A})\subseteq X \rightarrow X$ is an MLO, then $x$ is a
(reiteratively) $\tilde{X}$-distributionally irregular vector\index{$\tilde{X}$-distributionally irregular vector} for ${\mathcal A}$ iff
$x$ is a
(reiteratively) $\tilde{X}$-distributionally irregular vector\index{$\tilde{X}$-distributionally irregular vector} for the sequence 
$({\mathcal A}_{k}\equiv {\mathcal A}^{k})_{k\in {\mathbb N}};$ we accept this definition for all other parts (iii)-(xiv). 
\end{defn}

Keeping in mind the inequality $\overline{d}(A) \leq \overline{Bd}(A)$, 
it readily follows that the statements (A)-(C) and all implications clarified in Proposition \ref{funda} can be formulated for irregular vectors introduced above. 
Further on, we would like to note there are some important differences between Banach spaces and Fr\' echet spaces
with regard to the existence
of (reiteratively) distributionally unbounded vectors for sequences of MLOs:

\begin{example}\label{bf}
\begin{itemize}
\item[(i)] Suppose that the upper (Banach) density of set $\tilde{B}:=\{k\in {\mathbb N} : {\mathcal A}_{k}\mbox{ is purelly multivalued}\}$ is equal to $1,$ and $Y$ is a Banach space.
Then any vector $x\in \bigcap_{k=1}^{\infty}D({\mathcal A}_{k})$ is (reiteratively) distributionally unbounded. To see this, observe that in the part (ii)
of previous definition we can take $B=\tilde{B};$ then for any $k\in B,$ choosing arbitrary $x_{k}'\in {\mathcal A}_{k}x,$  we can always find $y_{k}' \in {\mathcal A}_{k}0$ such that we have
$\| x_{k}\|_{Y}=\|x_{k}'+y_{k}'\|_{Y}>2^{k},$ with $x_{k}=x_{k}'+y_{k}'.$
\item[(ii)] The situation is quite different in the case that $Y$ is a Fr\' echet space, we again assume that the set $\tilde{B}$ defined above has the upper (Banach) density equal to $1:$ Then there need not exist a vector $x\in \bigcap_{k=1}^{\infty}D({\mathcal A}_{k})$
that is (reiteratively) distributionally $m$-unbounded for some $m\in {\mathbb N}$. To illustrate this, consider the case in which $X:=Y:=C({\mathbb R}),$ equipped with the usual topology, and the operator ${\mathcal A}_{k}$ is defined by $D({\mathcal A}_{k}):=X$
and ${\mathcal A}_{k}f:=f+C_{[k,\infty)}({\mathbb R}),$ $k\in {\mathbb N},$ where $C_{[k,\infty)}({\mathbb R}) :=\{f \in C({\mathbb R}) : \mbox{supp}(f)\subseteq [k,\infty)\}.$ Then $\tilde{B}={\mathbb N}$ but for any $f\in X$
we have $\|f+g\|_{m}^{Y}=\|f\|_{m}^{Y}\equiv \sup_{x\in [-m,m]}|f(x)|,$
$g\in C_{[k,\infty)}({\mathbb R}),$ $m\leq k.$
\end{itemize}
\end{example}

Despite of the above, it should be noted that 
the existence of a scalar $\lambda \in \sigma_{p}({\mathcal A})$ with $|\lambda|>1$ implies that for any corresponding eigenvector $x\in X$ and any integer $k\in {\mathbb N}$ we have $\lambda^{k}x\in {\mathcal A}^{k}x,$ which in particular shows that $x$ has distributionally unbounded orbit under ${\mathcal A}.$ 

In \cite[Theorem 3.5]{kerry-drew}, we have proved that the hypercyclicity of an MLO ${\mathcal A}$ implies $\sigma_{p}({\mathcal A}^{\ast})=\emptyset.$ This is no longer true for 
dense Li-Yorke chaos, where we can state the following (see \cite[Proposition 11, Remark 12]{band} for single-valued case):

\begin{prop}\label{reci}
Suppose that ${\mathcal A}$ is an \emph{MLO} and $\lambda \in  \sigma_{p}({\mathcal A}^{\ast})$ satisfies $|\lambda|\geq 1.$ Then ${\mathcal A}$ cannot have a dense set of   
Li-Yorke near to zero 
vectors.
\end{prop}

\begin{proof}
Suppose the contrary, i.e., there exists a dense set $S$ of   
Li-Yorke near to zero 
vectors. Let $x^{\ast}\in X^{\ast} \setminus \{0\}$ be such that $\lambda x^{\ast}\in {\mathcal A}^{\ast}x^{\ast}.$ Then it can be simply shown that for each $x\in S$ and $n\in {\mathbb N}$ the supposition $x_{n}\in {\mathcal A}^{n}x$ implies 
\begin{align}\label{dov}
\bigl \langle x^{\ast},x_{n} \bigr \rangle =\bigl \langle \lambda^{n}  x^{\ast}, x \bigr \rangle.
\end{align}
Take now any $x\in S$ such that $\langle x^{\ast},x \rangle \neq 0.$ Then there exists a sequence $(x_{n})_{n\in {\mathbb N}}$ in $X$ such that $x_{n}\in {\mathcal A}^{n}x$ for all $n\in {\mathbb N}$ and $(x_{n})_{n\in {\mathbb N}}$ has a subsequence converging to zero. By \eqref{dov}, it readily follows that $|\lambda|<1,$ which is a contradiction.
\end{proof}

The following result is a kind of Godefroy-Shapiro and Dech-Schappacher-Webb Criterion for multivalued linear operators:

\begin{thm}\label{2.1ds-prim-444-skins} (cf. \cite[Theorem 3.8]{mendoza})
Suppose that $\Omega$ is an open connected
subset of ${\mathbb K}={\mathbb C}$ satisfying
$\Omega \ \cap \ S_{1} \neq \emptyset .$
Let $f : \Omega \rightarrow X \setminus
\{0\}$ be an analytic mapping such
that $\lambda f(\lambda)\in {\mathcal A}f(\lambda)$ for all $\lambda \in \Omega .$ Set
$\tilde{X}:=\overline{span\{f(\lambda) : \lambda \in
\Omega\}}.$ Then the operator ${\mathcal A}_{|\tilde{X}}$ is
topologically mixing in the space $\tilde{X}$ and the set of periodic points of ${\mathcal A}_{|\tilde{X}}$
is dense in $\tilde{X}.$
\end{thm}

Now we would like to propose the following problem:\vspace{0.2cm}

\noindent {\bf Problem 1.} Suppose that the requirements of Theorem \ref{2.1ds-prim-444-skins} hold true. Is it true that  
the operator ${\mathcal A}_{|\tilde{X}}$ is
densely distributionally chaotic in the space $\tilde{X}$? \vspace{0.1cm}

Assuming that the answer to Problem 1 is affirmative, we will be in a position to construct a substantially large class of densely distributionally chaotic MLOs (see e.g. \cite[Example 3.10, Example 3.12, Example 3.13]{kerry-drew}).

To state the next problem, let us assume that $T\in L(X)$ and there exists a dense linear submanifold $X_{0}$ of $X$ such that for each $x\in X_{0}$ one has $\lim_{n\rightarrow \infty}T^{n}x=0.$ Then it is well known that the existence of a distributionally unbounded vector $x$ for $T$ (a bounded sequence $(x_{n})$ in $X$ such that the sequence $(T^{n}x_{n})$ is unbounded) implies that there exists a dense distributionally irregular manifold (dense Li-Yorke irrregular manifold) for $T;$ see \cite[Theorem 15]{2013JFA} and \cite[Theorem 20]{band}.  
Now we would like to raise the following issue:\vspace{0.2cm}

\noindent {\bf Problem 2.} Do there exist similar conditions 
ensuring dense distributional chaos (dense Li-Yorke chaos) for orbits of MLOs?\vspace{0.1cm}

We continue by introducing the following notion:

\begin{defn}\label{idiotisen}
Let $\{0\} \neq X' \subseteq \tilde{X}$ be a linear manifold and let $i\in \{1,2\}$. Then
we say that:
\begin{itemize}
\item[(i)]
$X'$ is (reiteratively) $\tilde{X}$-distributionally
irregular manifold, resp. reiteratively $\tilde{X}$-distributionally
irregular manifold of type $i$/(strong) $\tilde{X}$-Li-Yorke irregular manifold for $({\mathcal A}_{k})_{k\in {\mathbb N}}$
((reiteratively) distributionally irregular manifold, resp. reiteratively distributionally irregular manifold of type $i$/(strong)  Li-Yorke irregular manifold in the case that $\tilde{X}=X$)
iff any element $x\in (X' \cap
\bigcap_{k=1}^{\infty}D({\mathcal A}_{k})) \setminus \{0\}$ is a
(reiteratively) $\tilde{X}$-distributionally irregular vector, resp. reiteratively $\tilde{X}$-distributionally irregular vector of type $i$/(strong) $\tilde{X}$-Li-Yorke irregular vector for
$({\mathcal A}_{k})_{k\in {\mathbb N}};$
\item[(ii)]
$X'$ is a uniformly (reiteratively) $\tilde{X}$-distributionally
irregular manifold\index{$\tilde{X}$-distributionally
irregular manifold!uniformly} for $({\mathcal A}_{k})_{k\in {\mathbb N}}$
(uniformly (reiteratively) distributionally irregular manifold\index{distributionally irregular manifold!uniformly} in the case that $\tilde{X}=X$)
iff there exists $m\in {\mathbb N}$ such that
any vector $x\in (X' \cap
\bigcap_{k=1}^{\infty}D({\mathcal A}_{k})) \setminus \{0\}$ is both (reiteratively) distributionally $m$-unbounded and (reiteratively) distributionally near to $0$ for $({\mathcal A}_{k})_{k\in {\mathbb N}}.$ 
\end{itemize}
The notions of a uniformly reiteratively $\tilde{X}$-distributionally
irregular manifold of type $i$ and a
uniformly $(\tilde{X},i)$-mixed irregular manifold for $i\in {\mathbb N}_{2}$ as well as
$(\tilde{X},i)$-mixed irregular manifold for $i\in {\mathbb N}_{4}$ and $\langle \tilde{X},i \rangle$-mixed irregular manifold for $i\in {\mathbb N}_{2}$ 
are introduced analogically.
The notion of any type of (uniformly) $\tilde{X}$-irregular manifold for an MLO ${\mathcal A} : D({\mathcal A})\subseteq X \rightarrow X$ is defined as before, by using the sequence $({\mathcal A}_{k}\equiv {\mathcal A}^{k})_{k\in {\mathbb N}}$.
\end{defn}

Let $i\in \{1,2\}$.
Using the elementary properties of metric, it can be simply verified that 
$X'$ is $2^{-m}_{\tilde{X}}$-(reiteratively) scrambled set for
$({\mathcal A}_{k})_{k\in {\mathbb N}}$ whenever $X'$ is a uniformly (reiteratively) $\tilde{X}$-distributionally
irregular manifold\index{$\tilde{X}$-distributionally
irregular manifold!uniformly} for $({\mathcal A}_{k})_{k\in {\mathbb N}};$ 
a similar notion holds for uniformly reiteratively $X'$-distributionally
irregular manifolds of type $i$ and uniformly $(\tilde{X},i)$-mixed irregular manifold for $i\in {\mathbb N}_{2}$. 
Clearly, if $X'$ is a (strong) $\tilde{X}$-Li-Yorke irregular manifold for $({\mathcal A}_{k})_{k\in {\mathbb N}},$ then $X'$ is a (strong) $\tilde{X}$-scrambled Li-Yorke set for $({\mathcal A}_{k})_{k\in {\mathbb N}};$ a similar statement holds for $(\tilde{X},i)$-mixed irregular chaos, where $i\in {\mathbb N}_{4},$ and $\langle \tilde{X},i \rangle$-mixed chaos, where $i\in {\mathbb N}_{2}.$ 
Furthermore, it can be
simply verified that, if $0\neq x\in \tilde{X} \cap \bigcap_{k=1}^{\infty}D({\mathcal A}_{k})$ is a (reiteratively)
$\tilde{X}$-distributionally irregular vector, resp. reiteratively
$\tilde{X}$-distributionally irregular vector of type $i$/(strong) $\tilde{X}$-Li-Yorke irregular vector for $({\mathcal A}_{k})_{k\in {\mathbb N}},$
then $X'\equiv span\{x\}$
is a uniformly (reiteratively) $\tilde{X}$-distributionally irregular manifold, resp. uniformly reiteratively $\tilde{X}$-distributionally irregular manifold of type $i$/(strong) $\tilde{X}$-Li-Yorke irregular manifold) for
$({\mathcal A}_{k})_{k\in {\mathbb N}};$ a similar statement holds for $(\tilde{X},i)$-mixed irregular chaos, where $i\in {\mathbb N}_{4},$ and $\langle \tilde{X},i \rangle$-mixed chaos, where $i\in {\mathbb N}_{2}.$ 

If $X'$ is dense in $\tilde{X},$
then the notions of dense (reiteratively) ($\tilde{X}$-)distributionally
irregular manifolds, dense uniformly (reiteratively) ($\tilde{X}$-)distributionally
irregular manifolds, and so forth, are defined analogically. The same agreements are accepted for all other types of chaos considered above.

If $({\mathcal A}_{k})_{k\in {\mathbb N}}$ and $\tilde{X}$ are given in advance, then we define the MLOs
${\mathbb A}_{k} : D({\mathbb A}_{k})\subseteq X \rightarrow Y$ by $D({\mathbb A}_{k}):=D({\mathcal A}_{k}) \cap \tilde{X}$ and
${\mathbb A}_{k}x:={\mathcal A}_{k}x,$ $x\in D({\mathbb A}_{k})$ ($k\in {\mathbb N}$). Then the following holds:

\begin{prop}\label{subspace-fric}
Let $i\in \{1,2\}.$
\begin{itemize}
\item[(i)] A vector $x$ is a (reiteratively) $\tilde{X}$-distributionally irregular
vector, resp. reiteratively $\tilde{X}$-distributionally irregular
vector of type $i$/(strong) $\tilde{X}$-Li-Yorke irregular vector for $({\mathcal A}_{k})_{k\in {\mathbb N}}$
iff $x$ is a (reiteratively) distributionally irregular vector, resp. reiteratively distributionally irregular vector of type $i$/(strong) Li-Yorke irregular vector for $({\mathbb A}_{k})_{k\in {\mathbb N}}.$ The same holds for $\langle \tilde{X},i \rangle$-mixed chaos and $(\tilde{X},j)$-mixed chaos, where $j\in {\mathbb N}_{4}.$
\item[(ii)] A linear manifold $X'$ is a (uniformly, (reiteratively)) $\tilde{X}$-distributionally irregular manifold, resp. (uniformly) reiteratively $\tilde{X}$-distributionally irregular manifold of type $i$/(strong) $\tilde{X}$-Li-Yorke irregular manifold for
$({\mathcal A}_{k})_{k\in {\mathbb N}}$ iff $X'$ is a (uniformly, (reiteratively)) distributionally irregular manifold, resp. (uniformly) reiteratively distributionally irregular manifold of type $i$/(strong) Li-Yorke irregular manifold for the sequence
$({\mathbb A}_{k})_{k\in {\mathbb N}}.$ The same holds for $(\tilde{X},i)$-mixed chaos.
\end{itemize}
\end{prop}

The fundamental distributionally chaotic properties of linear, not necessarily continuous, operators have been clarified in \cite[Corollary 3.12, Theorem 3.13]{kerry-drew}.
The proofs of these results,  which are intended solely for the analysis of single-valued operators, lean heavily on the methods and ideas from the theory of $C$-regularized semigroups (see \cite{knjigah}-\cite{knjigaho} and references cited therein for more details on the subject). For the investigations of distributionally chaotic properties of pure MLOs, we do not have such a powerful technique by now.

\section{Conclusions and final remarks}\label{finale}

In this paper, we have introduced a great number of distributionally chaotic and Li-Yorke chaotic properties for general sequences of binary relations acting between metric spaces. We have
carried out a special study of distributionally chaotic and Li-Yorke chaotic multivalued linear operators in Fr\' echet spaces, as well, providing a great number of illustrative examples and observations about problems considered.

In a series of recent research studies, N. C. Bernardes Jr. et al and T. Berm\'udez et al have analyzed the notions of mean Li-Yorke chaos, absolute Ces\`aro boundedness and 
Ces\`aro hypercyclicity for linear continuous operators in Banach spaces. We close the paper with the observation that these concepts can be analyzed for general sequences of binary relations over metric spaces.

\vspace{.1in}
\end{document}